\documentclass[a4paper,12pt] {article}
\usepackage{graphicx} 
\usepackage{amsmath,amsthm,amsfonts,amssymb,dsfont,mathtools,blindtext}
\usepackage[english]{babel}
\usepackage{hyperref}
\usepackage{fullpage}
\usepackage{graphicx}
\usepackage{caption}
\usepackage{float}

\newtheorem{theorem}{Theorem}[section]
\newtheorem{lemma}{Lemma}[section]
\newtheorem{corollary}{Corollary}[section]
\newtheorem{definition}{Definition}[section]
\newtheorem{proposition}{Proposition}[section]
\newtheorem{example}{Example}[section]

\title{Hamiltonian Systems as an Example of Invariant Measure}
\author{Lopes, Daniel F. \\ Universidade Federal de Minas Gerais}
\date{September 2025}

\begin{document}

\maketitle

\begin{abstract}
Hamiltonian systems are a classical example in the ergodic theory of flows with an invariant measure. In this matter, we present a brief introduction to measure theory and prove the Poincaré recurrence theorem to present the conditions for a system to be conservative. In the following, we discuss the Hamiltonian differential equations, vector fields, and their respective flows as an example of this invariance.
\end{abstract}

\section{Basics of Measure Theory}

This first chapter was based on definitions and theorems of basic measure theory that can be found on \cite{Krerley}, Chapter 1. We will present the foundation for the development of the next chapters.

\subsection{First definitions}
We first construct the definitions and properties of the measurable space and the measurement space.
 
\begin{definition} Let $M$ be a set. A subset algebra of M is a family $\mathcal{B}$ of subsets that contain $M$ and is closed under complement and union:
\end{definition}

\begin{enumerate}
    \item $M \subset \mathcal{B}$
    \item $A \in \mathcal{B} \Rightarrow A^{c} \in \mathcal{B}$
    \item $A \in \mathcal{B}$ and $B \in \mathcal{B} \Rightarrow A \cup B \in \mathcal{B}$
\end{enumerate}

 Clearly, the intersection $A \cap B$ of any $A,B \in \mathcal{B}$ is also contained in $\mathcal{B}$.
 
\begin{definition}
    A subset $\sigma$-algebra $\mathcal{B}$ of $M$ is an algebra that is also closed over the enumerable union.
\end{definition}

 Clearly, $\mathcal{B}$ is also closed to enumerable intersections.

\begin{definition}
    A measurable space is a par $(M,\mathcal{B})$ with $M$ being a set and $\mathcal{B}$ a subset $\sigma$-algebra of $M$. The elements of $\mathcal{B}$ are called \textit{measurable sets}.
\end{definition}

The intersection $\mathcal{B}$ of a nonempty family $\{B_i, i \in \mathcal{I}\}$ ($\mathcal{I}$ is an arbitrary set) of $\sigma$-algebras is also a $\sigma$-algebra. To construct the \textit{generated $\sigma$ algebra} of a family $\mathcal{E}$ of subsets, we imagine the intersection of all $\sigma$-algebras that contain $\mathcal{E}$. That is the smallest $\sigma$-algebra that contains $\mathcal{E}$.

\begin{definition}
A generated $\sigma$-algebra by $\mathcal{E}$ - a family of subsets of $M$ - is the smallest $\sigma$-algebra that contains $\mathcal{E}$.
\end{definition} 

 If M is a set and $\tau$ is the family of open subsets of $M$, we have $(M,\tau)$ as a topological space and a \textit{Borel $\sigma$-algebra}:

 \begin{definition}
 Let $(M,\tau)$ be a topological space. Then the Borel $\sigma$-algebra of M is the $\sigma$-algebra generated by $\tau$.
\end{definition}

 We now define the concept of measure and measurement space:

\begin{definition}
   A measure in a measurable space $(M,\mathcal{B})$ is a function $\mu : \mathcal{B} \to [0,\infty]$ (notice that, at principle, an infinite measure is included) that satisfies: 
\end{definition}

\begin{enumerate}
    \item $\mu(\emptyset) = 0$
    \item $\mu(\bigcup_{j=1}^\infty A_j) = \sum_{j=1}^\infty \mu(A_j)$ with $A_i \cap A_k = \emptyset, i \neq k$ ($\sigma$-aditivity)
\end{enumerate}

 The triple $(M,\mathcal{B},\mu)$ is called a \textit{measurement space}.\\
\indent A function $\mu : \mathcal{B} \to [0,\infty]$ is finitly aditive if:\\
\begin{center}
    $\mu(\displaystyle \bigcup_{j=1}^N A_j) = \displaystyle \sum_{j=1}^N \mu(A_j)$
\end{center}

\noindent for $A_1,\dots,A_N \in \mathcal{B}$ two by to disjoints.\\

 We now give a condition to $\sigma$-aditivity:\\

\begin{theorem} 
Let $\mathcal{B_0}$ be an algebra and $\mu_0 : \mathcal{B_0} \to [0,\infty]$ a finitely additive function with $\mu_0(M) = 1$, if:
\begin{center}
    $\displaystyle \lim_{n \to \infty} \mu(\displaystyle \bigcup_{j=1}^n A_j) = 0$
\end{center}
 for all sequences $A_1 \supset \dots\supset A_j \supset \dots$ of measurable sets such that $\bigcap_{j=1}^\infty A_j = \emptyset$, then $\mu_0$ is $\sigma$-additive.
\end{theorem}

\subsection{Lebesgue Measure}
 We now present a theorem on measure extension in order to construct Lesbegue measure, one of the most important measures in ergodic theory.

\begin{theorem}
Let $\mathcal{B}_0$ be a subset algebra of $M$ and $\mathcal{B}$ the $\sigma$-algebra generated by $\mathcal{B}_0$. Let 
$\mu_0 : \mathcal{B}_0 \to [0,\infty]$ be a finite additive function. Then, there is an unique finite additive function $\mu : \mathcal{B} \to [0,\infty]$ which is an extension of $\mu_0$ (in other words, $\mu$ restricted to $\mathcal{B}_0$ coincide with $\mu_0$).
\end{theorem}
 
 We first define the Lebesgue measure $\mu$ on the cube $M = [0,1]^d$, $d\geq1$, the following way: we call a \textit{rectangle} in $M$ any subset on the form $R = I_1\times\dots\times I_d$, each $I_j$ an interval, and define:\\

\begin{center}
    $\mu_0(R) = |I_1|\times\dots\times |I_d|$
\end{center}

 Then, consider the algebra $\mathcal{B}_0$ of $M$ in the form $B = R_1\cup\dots\cup R_N$, where $R_1,\dots,R_1$ two by two disjoints rectangles and define:
\begin{center}
    $\mu(B) = \mu_0(R_1) + \dots + \mu_0(R_N)$
\end{center}

 \noindent for all $B$ in this algebra. The Lebesgue measure on $M$ is the $\sigma$-algebra extension of $\mu_0$ generated by $\mathcal{B}$, which coincides with the Borel $\sigma$-algebra of $M$.\\
 Generalizing, we define the Lebesgue measure on $\mathbb{R}^{d}$ decomposing the space on unitary cubes:

\begin{center}
    $\mathbb{R}^{d} = \displaystyle \bigcup_{m_1 \in \mathbb{Z}}  \displaystyle \bigcup_{m_d \in \mathbb{Z}} [m_1,m_1 +1)\times\dots\times [m_d, m_d + 1)$
\end{center}

\noindent and defining, for a measure subset $E$:
\begin{center}
    $\mu(E) = \displaystyle \sum_{m_1 \in \mathbb{Z}} \dots \displaystyle \sum_{m_d \in \mathbb{Z}} [m_1,m_1 +1)\times\dots\times [m_d, m_d + 1)$ 
\end{center}

\indent Intuitively, the Lebesgue measure is what we commonly call volume, area or distance, depending on the given context.

\subsection{Fundamental Theorems}
We now present some fundamental theorems on measure theory and the notion of integration of a function with respect to a measure. Let $(M,\mathcal{B},\mu)$ be a measurement space.\\

\begin{definition}
    A measure in a measurable space $(M,\mathcal{B})$ is a function $\mu : \mathcal{B} \to [0,\infty]$ (notice that, at principle, an infinite measure is included) that satisfies:
\end{definition}

\begin{proposition}
Let $f_1,f_2$ be measurable functions and $c_1,c_2 \in \mathbb{R}$. Then are also measurable the following functions:
\begin{enumerate}
    \item $(c_1 f_1 + c_2 f_2)(x) = c_1 f_1(x) + c_2 f_2(x)$
    \item $(f_1 f_2)(x) = f_1(x)f_2(x)$
    \item $max\{f_1,f_2\}(x) = max\{f_1(x),f_2(x)\}$\\
\end{enumerate}
\end{proposition}

\begin{definition}
Let s be a function. Then the integral of s with respect to $\mu$ is given by;
\begin{center}
    $\displaystyle \int\ s\, d\mu = \displaystyle \sum_{j=1}^k \alpha_j \mu(A_j)$
\end{center}

\noindent with sets $\alpha_1,\dots, \alpha_k \in \mathbb{R}$, and $A_1,\dots,A_k \in \mathcal{B}$ two by two disjoints such that:
\[
s = \sum_{j=1}^k \alpha_j \chi_{A_j}, \quad 
\chi_A(x) = \begin{cases}
1, & x \in A, \\
0, & x \notin A
\end{cases}
\]
\end{definition}

\begin{theorem}
Let $f : M \to [-\infty,\infty]$ be a measurable function. Then there is a sequence $s_1, s_2,\dots$ of measurable functions such that:
\begin{center}
    $\displaystyle \lim_{k \to \infty} s_k(x) = f(x)$ for all $x \in M$
\end{center}
\end{theorem}

\begin{definition}
Let $f : M \to [0,\infty]$ be a non-negative measurable function, so:
\begin{center}
    $\displaystyle \int\ f\, d\mu = \displaystyle \lim_{n \to \infty} \displaystyle \int\ s_n\, d\mu$
\end{center}
 where $s_1 \leq s_2 \leq \dots$ is a sequence of crescent functions such that:
\begin{center}
    $\displaystyle \lim_{n \to \infty} s_n(x) = f(x)$ for all $x \in M$
\end{center}
\end{definition}

\begin{definition}
We say that a function is integrable if it is measurable and has finite integral.
\end{definition}

 Given a measurable function $f : M \to \mathbb{R}$ and a measurable set $E$, we define the integral of $f$ over $E$ by:
\begin{center}
    $\displaystyle \int_{E} f\, d\mu = \displaystyle \int f \chi_E\, d\mu$
\end{center}

 An important notion is that a property is valid \textit{in $\mu$-almost every point} if it is valid in all $M$ except possibly in a null measurable set.\\
 Given a measurable subset $A$ of $\mathbb{R}^{d}$, we say that $a \in A$ is a \textit{density point} of $A$ if this set fulfills the most part of any small neighborhood  of a:
\begin{center}
    $\displaystyle \lim_{\varepsilon \to 0} \dfrac{\mu(B(a,\varepsilon) \cap A)}{\mu(B(a,\varepsilon))} = 1$
\end{center}

\begin{theorem}
  Let $A$ be measurable subset of $\mathbb{R}^{d}$ with $\mu(A) > 0$. Then $\mu$-almost every point $a \in A$ is a density point of $A$.
\end{theorem}

\begin{theorem}[Dominated Convergence Theorem]
Let $f_n : M \to \mathbb{R}$ be a sequence of measurable functions and g an integrable function such that $|f_n(x)| \leq |g(x)|$ for $\mu$-almost every $x$ in $M$. If for $\mu$-almost every $x \in M$ the sequence $f_n(x)$ converges to $f(x)$, then f is integrable:
\begin{center}
    $\displaystyle \lim_{n \to \infty} \displaystyle \int f_n\, d\mu = \int f\, d\mu$
\end{center}
\end{theorem}

\section{Poincaré Recurrence Theorem}

We now present an important theorem on ergotic theory, with a structure of presentation based on \cite{Krerley}, Chapter 2.

\subsection{First Ideas}

By \textit{dynamical systems} we understand the transformations $f : M \to M$ under a metric or topological space $M$. Since it is hard to analyze \textit{all} points under some transformation, the main objective of ergodic theory is to study the  behavior of non-empty measure sets under some invariant measure.\\

 In this sense, the Poincaré recurrence theorem states that, for every \textit{finite invariant} measure, almost every point $x \in M$ is recurrent, which means its trajectory under a transformation comes closer to $x$ when the time goes to infinity.\\

 We will present and prove the metric version of the theorem, which is about measurable sets. The topological version, about open sets, is quite similar.

\subsection{Metric version}

Let $\mu$ be a measure defined in a Borel $\sigma$-algebra of a space $M$.

 We say that a measure $\mu$ is invariant over a transformation $f$ if:

\begin{center}
    $\mu(E) = \mu(f^{-1}(E))$ for all measurable sets $E \subset M$
\end{center}

 for flows:

\begin{center}
    $\mu(E) = \mu(f^{-t}(E))$ for all measurable sets $E \subset M$ and all $t \in \mathbb{R}$
\end{center}

\begin{theorem}\label{thm:main}
Let $f : M \in M$ be a measurable transformation and $\mu$ a finite invariant measure. Let $E \subset M$ be any measurable set with $\mu(E) > 0$. So, $\mu$-almost every point $x \in E$ has some iterated $f^n(x)$, $n>1$ that is in $E$.
\end{theorem}

 Before proving the theorem itself, we will prove a strong consequence of it:

\begin{corollary}
    Given Theorem 2.1, for $\mu$-almost every point $x \in E$ there exists an infinite number of values of $n\geq1$ such that $f^n(x)$ is in $E$.
\end{corollary}

\begin{proof} Let, for each $k\geq1$, $E_k$ be the set of points $x \in E$ that return exactly $k$ times to $E$. This means that, for a $x \in E_k$, there exists exactly $k$ values of iterates $n\geq1$ such that $f^n(x) \in E$.

 The set of points that return to $E$ a finite number of times is:

\begin{center}
    $\displaystyle \bigcup_{k=1}^\infty E_k$
\end{center}

 Thus, it is enough to prove that $\mu(E_k) = 0$ for all $k\geq 1$. We will prove it by contradiction.

 Suppose that $\mu(E_k) > 0$ for some $k\geq1$. The, by the previous theorem, we got that almost every point $x \in E_k$ has an iterate $f^n(x)$ in $E_k$. Let's fixate $x$ and let $y = f^n(x)$. Since $y \in E_k$, $y$ has exactly $k$ future iterates in $E$. But, since $y$ is an iterate of $x$, $x$ has $k+1$ future iterates in $E$. That contradicts the fact that $x \in E_k$. Thus, $\mu(E_k) = 0$ for all $k\geq1.$ 
\end{proof}

 Now, to the main proof:

\begin{proof}[Proof of Theorem \ref{thm:main}]
Let $E^0$ be the set of points $x \in E$ that never return to $E$. The objective is to prove that $\mu(E^0) = 0$. First, let state that the pre-images $f^{-n}(E^0)$ are two by two disjoints. Suppose by contradiction that there exist $m>n\geq1$ such that $f^{-m}(E_0)$ intersects $f^{-n}(E_0)$. Let $x \in f^{-m}(E_0) \cap f^{-n}(E_0)$ and $y = f^n(x)$. So, $y$ is in $E^0$ and $f^m(x) = f^{m-n}(f^n(x)) = f^{m-n}(y)$. Since $f^m(x) \in E^0$ and $E^0 \subset E$, $y$ returns to $E$ at least once ($m-n>0$). That contradicts the definition of $E^0$. So the pre-images are two by two disjoints, in fact.

 That implies:

\begin{center}
    $\mu(\displaystyle \bigcup_{n=0}^\infty f^{-n}(E_0)) = \displaystyle \sum_{n=0}^\infty \mu(f^{-n}(E^0)) = \displaystyle \sum_{n=0}^\infty \mu(E^0)$
\end{center}

 In the last equality we used the hypothesis that $\mu$ is invariant, so $\mu(f^{-n}(E^0)) = \mu(E^0)$ for all $n\geq1$. Since the measure is finite, the left side is finite. But the right side is an infinite sum of equal terms. The only way it can be finite is by $\mu(E^0) = 0$.
\end{proof}

\section{Conservative Systems}

We now present and characterize a \textit{conservative} system in the sense that was proposed as an example on \cite{Krerley}, Section 3.2.

\subsection{Invariant Measures}

We now present a proposition characterizing an invariant measure over a transformation or a flow:

\begin{proposition}
    Let $f : M \to M$ be a transformation and $\mu$ a measure. Then, $f$ preserves $\mu$, if and only if, for every integrable function $\varphi : M \to \mathbb{R}$ it is valid: 

\begin{center}
    $\displaystyle \int \varphi\ d\mu = \displaystyle \int \varphi \circ f\ d\mu$
\end{center}
\end{proposition}

\begin{proof}[Proof. \( (\Rightarrow) \)]
Assume that $f$ preserves $\mu$. If $\mu$ is the characteristic function of a set $A$, then:

\begin{align}
\chi_{f^{-1}(A)} &= \chi_A \circ f \notag \\
\chi_{f^{-1}(A)} &= \varphi \circ f
\end{align}

 Since:
\begin{center}
    $\mu(f^{-1}(A)) = \displaystyle \int \chi_{f^{-1}(A)}\ d\mu$
\end{center}

 By $(1)$:
\begin{center}
    $\mu(f^{-1}(A)) = \displaystyle \int \varphi \circ f\ d\mu$
\end{center}

 Using the fact that $\mu(f^{-1}(A)) = \mu(A)$:

\begin{align*}
    \mu(A) &= \displaystyle \int \varphi \circ f\ d\mu\\
    \displaystyle \int \chi_A\ d\mu &= \displaystyle \int \varphi \circ f\ d\mu\\
    \displaystyle \int \varphi\ d\mu &= \displaystyle \int \varphi \circ f\ d\mu
\end{align*}

 Then, it is proved when $\varphi$ is a characteristic function. By the linearity of the integral, if $\varphi$ is a simple function (which means, a linear combination of characteristic functions $\chi_{A_k}$ of disjoint sets $A_1,\dots,A_k$), then the equality holds. Finally, if $\varphi$ is an arbitrary function, by the integral definition:

\begin{center}
    $\displaystyle \int \varphi\ d\mu = \lim_{n \to \infty} \displaystyle \int \varphi_n\ d\mu$
\end{center}

\noindent where $\varphi_n$ are simple crescent functions and $\varphi_n \to \varphi$. Still, $\varphi_n \circ f$ are simple crescent functions and $\varphi_n \circ f \to \varphi \circ f$. Then:

\begin{center}
    $\displaystyle \int \varphi \circ f\ d\mu = \lim_{n \to \infty} \displaystyle \int \varphi_n \circ f\ d\mu$
\end{center}

 Since $\int \varphi_n\ d\mu = \int \varphi_n \circ f\ d\mu$, taking the limit on both sides:

\[
\displaystyle \int \varphi\ d\mu = \displaystyle \int \varphi \circ f\ d\mu 
\]
\end{proof}

\begin{proof}[Proof. \( (\Leftarrow) \)]Given a set $A$ in a Borel $\sigma$-algebra of $M$, let $\varphi = \chi_A$:

\begin{align*}
    \displaystyle \int \varphi\ d\mu &= \displaystyle \int \varphi \circ f\ d\mu\\[2mm]
    \mu(A) &= \mu(f^{-1}(A))
\end{align*}
\end{proof}

\begin{figure}[h]
\nopagebreak
\subsection{Characterization of Conservative Systems}

Let $U$ be an open set on $\mathbb{R}^d$, $d \geq 1$ and let $f : U \to U$ a $C^1$ diffeomorphism. We represent by $m$ the Lebesgue measure volume in $\mathbb{R}^k$:
\end{figure}

\begin{center}
    $m(B) = \displaystyle \int_B dx_1\dots dx_d$ and $\displaystyle \int_B \phi\ dvol = \displaystyle \int_B \phi(x_1,\dots x_d)\ dx_1\dots dx_d$
\end{center}

\noindent for any measurable set $B$ and any integrable function $\varphi$.\\

 The change of variables formula stats that, for any measurable set $B \subset U$:

\begin{center}
    $m(f(B)) = \displaystyle \int_B |\det Df|\ dm$
\end{center}

 Then we can deduce the following:\\

\begin{lemma}
    A $C^1$ diffeomorphism $f : M \to M$ preserves the volume if and only if $|\det Df| = 1$.
\end{lemma}

\begin{proof}[Proof. \( (\Rightarrow) \)] 
Let $E$ be a measurable set and $B = f^{-1}(E)$. If $|\det Df| = 1$, then:

\[
m(E) = \displaystyle \int_B 1\ dm = m(B) = m(f^{-1}(E))
\]
\end{proof}

\begin{proof}[Proof. \( (\Leftarrow) \)]Suppose $|\det Df| > 1$ at some point $x$. Since the Jacobian is continuous, there is a neighborhood $U$ of $x$ and some $\sigma > 1$ such that:

\begin{center}
    $|\det Df(y)| \geq \sigma$ for all $y \in U$ 
\end{center}

 The, by making $B = U$:

\begin{center}
    $m(f(U)) \geq \displaystyle \int_B \sigma\ dm \geq \sigma m(U) > m(U)$
\end{center}

 Denoting $E = f(U)$, then $m(E) > m(f^{-1}(E)$ and the volume is not invariant by f. It is similar for $|\det Df| < 1$ at some point.
\end{proof}

 We now analyze the case of flows $\phi^t : U \to U$, $t \in \mathbb{R}$. Suppose a $C^1$ flow. The Lemma 3.1 is applicable: the flow preserves the volume if and only if:

\begin{center}
    $\det D\phi^t(x) = 1$ for all $x \in U$ and for all $t \in \mathbb{R}$
\end{center}

 Suppose the flow $\phi^t$ as the trajectories of a $C^1$ vector field $F : U \to U$, which means, $\phi^t(x)$ is the solution of the differential equation:

\begin{center}
    $\dfrac{dx}{dt} = F(x)$
\end{center}

\noindent at time $t$. The Liuoville formula gives the Jacobian of $\phi^t$ in terms of $\operatorname{div} \textbf{F}$:

\begin{center}
        $\det Df^t(x) = \exp{\left(\displaystyle \int_0^t \operatorname{div} \textbf{F}(\phi^s(x))\ ds\right)}$
\end{center}

 For volume invariance, $\det Df^t(x) = 1$, then $\operatorname{div} \textbf{F} = 0$:\\

\begin{lemma}
    A flow $\phi^t$ associated with a $C^1$ vector field $F$ preserves the volume if and only if $\operatorname{div}\textbf{F} = 0$.
\end{lemma}

 By the Poincaré recurrence theorem for flows, if $U$ has finite measure and $\operatorname{div} \textbf{F} = 0$, then $\mu$-almost every point $x$ is invariant by the flow $\phi^t$ of $F$.

\section{Hamiltonian Systems}

In this last chapter, we discuss the Hamiltonian systems, their differential equations, vector fields and flows, and prove their volume invariance on phase space. The Section 4.1 was based  on \cite{Robinson}, Section 6.1, and the second part is a deepening discussion of two examples briefly presented on \cite{Robinson}, also on Section 6.1.

\subsection{Hamiltonian Differential Equations}

If we have a particle moving through a configuration space $S = \{\mathbf{q} \in \mathbb{R}^n\}$, we can describe the system by determining the position $\mathbf{q}$ and the momentum $\mathbf{p}$ of the particle over time $t \in \mathbb{R}$. In this sense, there exists the potential energy $V : S \to \mathbb{R}$, that depends on $\mathbf{q}$, the kinetic energy $T : S \to \mathbb{R}$, that depends on $\mathbf{p}$, and the Hamiltonian:

\begin{center}
    $H(\mathbf{q},\mathbf{p}) = V(\mathbf{q}) + T(\mathbf{p})$
\end{center}

 That depends on both $\mathbf{q}$ and $\mathbf{p}$, and represents the total energy of the phase space (the space of all possible combinations of $\mathbf{q}$ and $\mathbf{p}$) $S \times \mathbb{R}$. It is natural, then, to imagine that some of this measures would help to study the behavior of $\mathbf{q}$ and $\mathbf{p}$ over time, namely, the Hamiltonian $H$.\\

 First, by the equation of motion:

\begin{equation}
    m\ddot{q} = -\nabla V_q
\end{equation}

\noindent and by the definition of momentum $\mathbf{p} = m\mathbf{v}$, we have:

\begin{align}
    \dot{q_j} &= \frac{p_j}{m} \notag \\
    \dot{p_j} &= m\ddot{q}_j
\end{align}

 By $(2)$:

\begin{equation}
    \dot{p_j} = -\dfrac{\partial V}{\partial q_j}
\end{equation}

\noindent for $1 \leq j \leq n$

 Equations $(2)$ and $(3)$ can be written in the Hamiltonian form. First, deriving $H$ in terms of $q_j$:

\begin{center}
    $\dfrac{\partial H}{\partial q_j} = \dfrac{\partial V}{\partial q_j} + \dfrac{\partial T}{\partial q_j}$
\end{center}

 Since $\partial T/\partial q_j = 0$, by $(3)$:

\begin{center}
    $\dot{p_j} = -\dfrac{\partial H}{\partial q_j}$
\end{center}

 Deriving in terms of $p_j$, remembering that $T(\mathbf{p}) = \mathbf{p} \cdot \mathbf{p}/2m$:

\begin{center}
    $\dfrac{\partial H}{\partial p_j} = \dfrac{\partial V}{\partial p_j} + \dfrac{\partial}{\partial p_j} \left(\dfrac{1}{2m}\displaystyle \sum_{i = 1}^n p_i^2\right)$
\end{center}

 Since $\partial V/\partial p_j = 0$:

\begin{center}
    $\dfrac{\partial H}{\partial p_j} = \dfrac{p_j}{m} \Rightarrow \dot{q}_j = \dfrac{\partial H}{\partial p_j}$
\end{center}

 Then, we have the system:

\[
\left\{
\begin{aligned}
\dot{q}_j &= \frac{\partial H}{\partial p_j}\\[2mm]
\dot{p}_j &= -\frac{\partial H}{\partial q_j}
\end{aligned}
\right., \; 1 \le j \le n
\]

 If we consider the phase space as $\mathbb{R}^{2n}$, the time-dependent Hamiltonian is $H : \mathbb{R}^{2n} \times \mathbb{R} \to \mathbb{R}$, $H(q_1,\dots,q_n,p_1,\dots p_n,t)$, with $t \in \mathbb{R}$ being the time. Then, we have the system of differential equations:

\[
\left\{
\begin{aligned}
\dot{q}_j &= \frac{\partial H(\mathbf{q},\mathbf{p},t)}{\partial p_j}\\[2mm]
\dot{p}_j &= -\frac{\partial H(\mathbf{q},\mathbf{p},t)}{\partial q_j}
\end{aligned}
\right., \; 1 \le j \le n
\]

 The Hamiltonian vector field $X_H$ for this system is:

\begin{center}
    $X_H(\mathbf{q},\mathbf{p}) = \left(\dfrac{\partial H}{\partial q_1},\dots,\dfrac{\partial H}{\partial q_n},-\dfrac{\partial H}{\partial p_1},\dots,-\dfrac{\partial H}{\partial p_n}\right)$
\end{center}

 If $\mathbf{z} = (\mathbf{q},\mathbf{p}) = (z_1,\dots,z_{2n})$, then:

\begin{center}
    $X_H(\mathbf{z}) = (\dot{z_1},\dots,\dot{z_{2n}})$
\end{center}

 In time-dependent case, the flow of the Hamiltonian vector field for a point $\mathbf{z} = (q_1,\dots,q_n,p_1,\dots,p_n)$ in the phase space $\mathbb{R}^{2n}$ at the time $t$ is:

\begin{center}
    $\phi^t(\mathbf{z}) = (q_1(t),\dots,q_n(t),p_1(t),\dots,p_n(t))$
\end{center}

\noindent where $(\mathbf{q}(t),\mathbf{p}(t))$ satisfies the Hamiltonian systems of equations at time $t$, which means, the flow takes the initial point $\phi^0(\mathbf{q},\mathbf{p})$ to $\mathbf{z}(t)$ following the vector field  $X_H$.

\begin{proposition}
    Let $H$ be a time-independent Hamiltonian for a Hamiltonian vector field $X_H$. Then, $H$ is constant along solutions of $X_H$.
\end{proposition}

\begin{proof}
\[
\dot{H}(\mathbf{q}(t),\mathbf{p}(t)) = \displaystyle \sum_{j=1}^n \left(\dfrac{\partial H}{\partial q_j}q_j + \dfrac{\partial H}{\partial p_j}p_j\right) = \displaystyle \sum_{j=1}^n \dfrac{\partial H}{\partial q_j}\dfrac{\partial H}{\partial p_j} + \dfrac{\partial H}{\partial p_j}\left(-\dfrac{\partial H}{\partial q_j}\right) = 0
\]
\end{proof}

\begin{proposition}
    A Hamiltonian vector field has zero divergence, so its flow preserves the volume.
\end{proposition}

\begin{proof}
\[
\operatorname{div}(X_H)(\mathbf{z}) = \displaystyle \sum_{j=1}^n \dfrac{\partial}{\partial q_j}\left(\dfrac{\partial H}{\partial p_j}\right)(\mathbf{z}) + \dfrac{\partial}{\partial p_j}\left(-\dfrac{\partial H}{\partial q_j}\right)(\mathbf{z}) = 0 
\]
\end{proof}

\subsection{Examples and conclusion}

In this section we will analyze two important examples of Hamiltonian systems, depicting their Hamiltonian function, the solutions $(p,q)$ and $(\theta,p)$ of the equations and the phase portraits (the geometric representation of the solutions). Finally, we present our conclusions about the topics discussed.\\

\begin{example} 
The equation of the linear harmonic oscillator is:

\begin{center}
    $m\ddot{q} +kq = 0$
\end{center}

\noindent with $k = m\omega^2$, we have:

\begin{equation}
\ddot{q} + \omega^2q = 0
\end{equation}

 The kinetic energy is:

\begin{center}
    $T(p) = \dfrac{p^2}{2m}$
\end{center}

\noindent and the potential energy is:

\begin{align*}
V(q) &= - \int m \ddot{q}\, dq \\
V(q) &= - \int -\omega^2 q\, dq
\end{align*}

 Making the constant $C = 0$, such that $V(0) = 0$, it comes:

\begin{center}
    $V(q) = \dfrac{m\omega^2q^2}{2}$
\end{center}

 Thus, the Hamiltonian function is:

\begin{center}
    $H(q,p) = \dfrac{p^2}{2m} + \dfrac{m\omega^2q^2}{2}$
\end{center}

 Choosing $m=1$ for simplicity:

\begin{center}
    $H(q,p) = \dfrac{p^2 + \omega^2p^2}{2}$
\end{center}

 The solutions of $(5)$ can be found by the characteristic equations and are the following:
\begin{align*}
q(t) &= A\cos(\omega(t - \delta)) \\[2mm]
p(t) &= -A\omega\sin(\omega(t - \delta))
\end{align*}

 Since $H$ is time independent, we can fixate $H = E$:
\begin{align*}
\dfrac{p^2}{2m} + \dfrac{m\omega^2q^2}{2} = E\\[2mm]
\dfrac{p^2}{2mE} + \dfrac{q^2}{\left(\dfrac{2E}{m\omega^2}\right)} = 1
\end{align*}

 That means that the phase portrait of $(q,p)$ represents concentric ellipses depending on $E$. Normalizing ($m = 1$, $\omega = 1$), we have concentric circles depending on $E$:

\begin{figure}[H]
        \centering
        \includegraphics[width=0.5\linewidth]{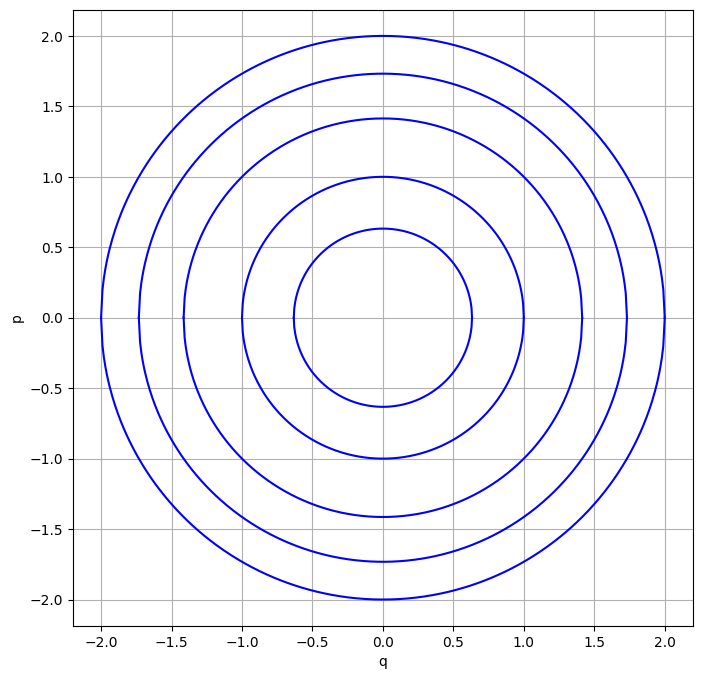}
        \caption{Linear harmonic oscillator phase portrait}
        \label{fig:placeholder}
\end{figure}
\end{example}

\begin{example}
The equation of a pendulum of mass m and length L is:

\begin{equation}
mL\ddot{\theta} = -mg\sin \theta
\end{equation}

 The potential energy is:

\begin{align*}
    V(\theta) =& - \displaystyle \int mL^2\ddot\theta\ d\theta\\
    V(\theta) =& - \displaystyle \int - mgL\sin\theta\ d\theta\\
    V(\theta) =& -mgL\cos\theta + C
\end{align*}

 Choosing $C = mgL$, such that $V(0) = 0$:

\begin{center}
    $V(\theta) = mgL(1 - cos\theta)$
\end{center}

 Dividing by $mL^2$ for simplification:

\begin{center}
    $V(\theta) = \dfrac{g}{L}(1-\cos\theta)$
\end{center}

 So, we got the normalized Hamiltonian:

\begin{center}
    $H(\theta,p) = \dfrac{p^2}{2} + \dfrac{g}{L}(1-\cos\theta)$
\end{center}

 In the case of small oscillations ($\sin\theta \approx \theta$), the pendulum behaves like a linear oscillator and its phase portrait is the same (concentric ellipses depending on $E$). The solutions are also similar:
\begin{align*}
    \theta(t) &= \theta_{max}\cos(\omega t - \phi)\\
    p(t) &= -\theta_{max}\omega\sin(\omega t - \phi)
\end{align*}

 For greater amplitudes, we have:

\begin{center}
    $p(t) = \pm \sqrt{2(E - (1 - \cos\theta(t)))}$
\end{center}

 and a different phase portrait:

\begin{figure}[H]
    \centering
    \includegraphics[width=0.8\linewidth]{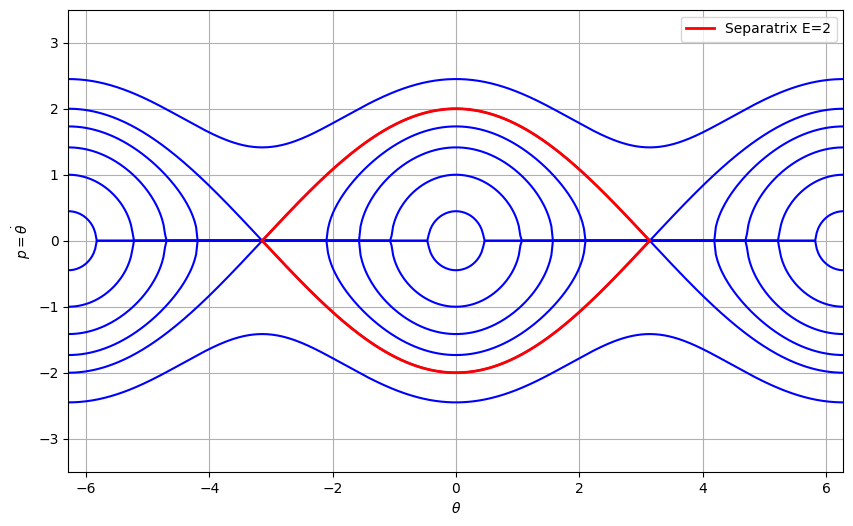}
    \caption{Non-linear pendulum phase portrait}
    \label{fig:placeholder}
\end{figure}

\begin{itemize}
    \item If $E < 2$, we have closed trajectories (represents the oscillations)
    \item If $E > 2$, we have open curves (represents complete rotations)
\end{itemize}
\end{example}

 In summary, Hamiltonian systems are an example of conservative systems, since their flow $\phi^t$ preserves the volume (in terms of Lebesgue measure) in phase space, or equivalently, the divergence of its vector field $X_H$ equals zero. This property can be visualized in the examples we discussed. Since they are linear, their phase space is $\mathbb{R}^2$ (the phase portrait). If we take any nonempty set $U \subset \mathbb{R}^2$, we have $\mu(U) = \mu(\phi^t(U))$, for all $t \in \mathbb{R}$.

\bibliographystyle{plain}
\bibliography{referencias}

\begin{thebibliography}{1}

\bibitem{Krerley}
Krerley Oliveira.
\newblock Um primeiro curso sobre teoria ergódica com aplicações.
\newblock {\em Publicações Matemáticas}, 2005.

\bibitem{Robinson}
Clark Robinson.
\newblock {\em Dynamical Systems: Stability, Symbolic Dynamics, and Chaos}.
\newblock CRC Press, 1999.

\end{thebibliography}

\end{document}